\documentclass[10pt,reqno]{amsart}
\usepackage[latin1]{inputenc}
\usepackage{amsmath,amssymb,amsthm,
	verbatim,
	fancyhdr,amsfonts}

\usepackage[format=plain,labelfont=it,textfont=it]{caption}
\usepackage{needspace}
\usepackage{float}
\usepackage{tikz}
\usetikzlibrary{decorations.pathreplacing}

\usepackage[pdftex]{lscape}
\usepackage{amsmath}
\usepackage{amsthm}
\usepackage{amsfonts,amsmath,amssymb}
\usepackage{enumerate,verbatim,color}
\usepackage{epsfig}
\usepackage{textcomp}
\usepackage{graphicx}
\usepackage{amsfonts}
\usepackage{mathrsfs}
\usepackage{upgreek}
\usepackage{mathtools}
\usepackage{pdfpages}
\usepackage{fullpage}

\newcommand{\virgolette}[1]{``#1''}
\newtheorem{teorema}{Theorem}[section]
\newtheorem{coro}[teorema]{Corollary}
\newtheorem{lemma}[teorema]{Lemma}
\newtheorem{prop}[teorema]{Proposition}
\newtheorem{osss}[teorema]{Remark}

\newtheorem*{oss}{Remark}
\theoremstyle{definition}
\newtheorem{defi}[teorema]{Definition}
\theoremstyle{remark}
\newtheorem*{assumptionI}{\bf Assumptions I}
\newtheorem*{assumptionII}{\bf Assumptions II}
\newtheorem*{assumptionA}{\bf Assumptions A}

\newcommand{\iii}{{\, \vert\kern-0.25ex\vert\kern-0.25ex\vert\, }}

\newcommand{\ffi}{\varphi}

\newcommand{\LL}{\mathcal{L}}
\newcommand{\MM}{\mathcal{M}}

\newcommand{\NN}{\mathcal{N}}
\newcommand{\KK}{\mathcal{K}}

\newcommand{\Di}{\mathcal{D}}

\newcommand{\N}{\mathbb{N}}
\newcommand{\R}{\mathbb{R}}
\newcommand{\Q}{\mathbb{Q}}
\newcommand{\C}{\mathbb{C}}

\newcommand{\Hi}{\mathscr{H}}

\newcommand{\Gi}{\mathscr{G}}

\newcommand{\Si}{\mathscr{S}}

\newcommand{\dd}{\partial}

\newcommand{\ra}{\rangle}
\newcommand{\la}{\langle}

\newcommand{\G}{\Gamma}
\newcommand{\Ti}{\mathcal{T}}

\begin{document}

\date{}
\title[Bilinear control]{Controllability of periodic bilinear quantum systems on infinite graphs }

\author{Ka\"{\i}s Ammari}
\address{UR Analysis and Control of PDEs, UR 13ES64, Department of Mathematics, Faculty of Sciences of Monastir, University of Monastir, Tunisia}
\email{kais.ammari@fsm.rnu.tn} 

\author{Alessandro Duca}
\address{Institut Fourier, Université Grenoble Alpes, 100 Rue des Mathématiques, 38610 Gières, France} 
\email{alessandro.duca@unito.it}

\begin{abstract}
In this work, we study the controllability of the bilinear Schr\"odinger equation on infinite graphs for periodic 
quantum states. We consider the bilinear Schr\"odinger equation \eqref{mainx1} $i\dd_t\psi=-\Delta\psi+u(t)B\psi$ 
in the Hilbert space $L^2_p$ composed by functions defined on an infinite graph $\Gi$ verifying periodic boundary 
conditions on the infinite edges. The Laplacian $-\Delta$ is equipped with specific boundary conditions, $B$ is a 
bounded symmetric operator and $u\in L^2((0,T),\R)$ with $T>0$. We present the well-posedness of the \eqref{mainx1} 
in suitable subspaces of $D(|\Delta|^{3/2})$ . In such spaces, we study the global exact controllability and we 
provide examples involving tadpole graphs and star graphs with infinite spokes.  
\end{abstract}


\maketitle

\section{Introduction}\label{intro}

Graph type structures (Figure 
\ref{fig:1}) have been widely studied for the modeling of 
phenomena arising in science, social 
sciences 
and
engineering. Among the many applications to quantum mechanics, they were used to study 
the dynamics of free
electrons in organic molecules starting from the seminal work \cite{198}, the 
superconductivity in granular and 
artificial materials \cite{121212}, acoustic and electromagnetic wave-guides networks in \cite{112, 181}, etc.

\begin{figure}[H]
	\centering
	\includegraphics[width=\textwidth-80pt,height=30pt]{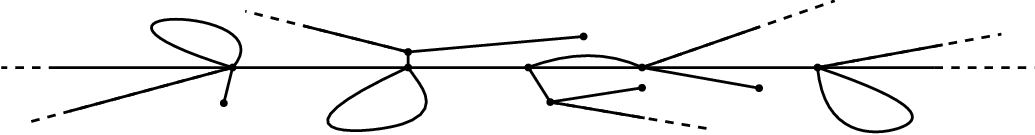}
	\caption{An infinite graph is an one-dimensional domain composed by vertices (points) connected by edges 
(segments and half-lines).}\label{fig:1}
\end{figure}

We consider a particle trapped on a network of wave-guides or wires where 
some 
branches are way longer than the others. We model the long branches with half-lines and the remaining ones with 
segments in order to represent the network by an infinite graph. The nodes of the 
network are ideal so that the crossing particle is 
subjected to zero resistance during the motion and we assume that the system is subjected to an external 
field which plays the role of control.

\smallskip

A natural choice for such setting is to represent the network by an infinite graph $\Gi$ and the state of the 
particle by a function $\psi$ with domain $\Gi$. The state $\psi$ belongs to a suitable
Hilbert space $\Hi$ and the dynamics of the particle is modeled by the bilinear 
Schr\"odinger equation in $\Hi$ \begin{align}\label{eq1}i\dd_t\psi(t,x)= A\psi(t,x) +u(t) B\psi(t,x),\ \ \ \ \ \ 
t\in 
[0,T],\ x\in \Gi,\end{align} where $A$ is a positive self-adjoint operator. The term $u(t)B$ 
represents the time 
dependent external field acting on the system which action is given by the bounded symmetric 
operator $B$ and its intensity by the control function 
$u\in L^2((0,T),\R)$. 

\smallskip

In this work, we consider $\Hi$ as the Hilbert space composed by $L^2_{loc}$ functions over the graph satisfying 
periodic boundary conditions on the infinite edges and $A$ is a Laplacian equipped with suitable boundary 
conditions. We study the controllability of the bilinear Schr\"odinger equation \eqref{eq1} according to the 
choice of the 
graph. Our purpose is to analyze when it is possible to control exactly the motion by 
time-varying the intensity of the external field.

\medskip

{\noindent \bf \underline{Some biblio}g\underline{ra}p\underline{h}y}

\smallskip

The mathematical analysis of operators defined on networks was
preliminarily addressed in \cite{209} by Ruedenber and Scherr. In this work, they studied the dynamics of 
particular electrons in the conjugated double-bounds organic molecules. These particles move
as if they were trapped on a
network of wave-guides and the graphs are obtained 
as the
idealization of such structures in the limit where the diameter of the section is much smaller 
than the length. A similar approach was developed by Saito in \cite{211, 212} where the graphs are obtained as 
``shrinking'' domains. For analogous ideas, we refer to the papers \cite{202, 208}.

\smallskip

The controllability of finite-dimensional quantum systems modeled by equations as \eqref{eq1},
when $A$ and $B$ are 
$N\times N$ Hermitian matrices, is well-known for being  
linked to the rank of the Lie algebra spanned by $A$ and $B$ (see \cite{basso,corona}). Nevertheless, the Lie 
algebra rank condition can not be used for infinite-dimensional quantum systems (see \cite{corona}).

\smallskip

The {global approximate controllability} of the bilinear Schr\"odinger equations \eqref{eq1} was proved with 
different techniques in 
literature. We refer to \cite{milo,nerse2} for Lyapunov techniques, to \cite{ugo2,ugo3} for 
adiabatic arguments and to \cite{ugo,nabile} for Lie-Galerking methods.

\smallskip

The {exact controllability} of infinite-dimensional quantum systems is in general more delicate. For instance, the 
controllability and observability of the linear Schr\"odinger equation are reciprocally 
dual. Various results were developed by addressing directly or by duality the control problem with multiplier 
methods \cite{lion,Machty}, 
microlocal analysis \cite{rauch,burqa,lollo} and 
Carleman estimates \cite{44,lasiecka,330}. However, a complete 
theory on networks
is 
far from being formulated. Indeed, the interaction between the different components of the structure may generate 
unexpected phenomena. For further details on the subject, we refer to \cite{wave}.

\smallskip

An important property of the bilinear Schr\"odinger equation is that its controllability can not be approached with 
the techniques valid for the 
linear Schr\"odinger equation. Indeed, the dynamics of \eqref{eq1} 
is well-known for not being exactly controllable in the Hilbert space $\Hi$
where it is defined when $B$ is a bounded operator and $u\in L^2((0,T),\R)$ with $T>0$ (even though it is 
well-posed in such space). This result was proved by Turinici in \cite{torino} by exploiting the theory developed 
by Ball, Mardsen and Slemrod in \cite{ball} (see 
\cite{bal, bal2} for other results on bilinear systems).
As a consequence, the classical 
techniques can not be exploited for the exact controllability of bilinear quantum systems.

\smallskip

The turning point for this kind of studies has been the idea of controlling the equation in specific subspaces of 
$D(A)$. 
Preliminarily introduced by Beauchard in \cite{be1}, this approach was mostly popularized by the work 
\cite{laurent} of Beauchard and Laurent. There, they considered the bilinear 
Schr\"odinger equation $\eqref{eq1}$ on the interval $\Gi=(0,1)$ when $\Hi=L^2((0,1),\C)$, $B$ is a suitable 
multiplication operator 
and $A=-\Delta_D$ is the Dirichlet Laplacian
$$D(-\Delta_D)=H^2((0,1),\C)\cap H^1_0((0,1),\C)),\ \ \ \ \ -\Delta_D\psi:=-\Delta\psi,\ \ \ \ \forall\psi\in 
D(-\Delta_D).$$ They proved the {well-posedness} and the {local 
exact 
controllability} of the equation in the space $D(|\Delta_D|^{3/2})$. Afterwards, different works on the subject 
were developed. We refer to 
\cite{beauchard2017,mio2} for {global exact controllability} results and
\cite{mio2,morgane1,morganerse2} for simultaneous exact controllability results.

\smallskip

The controllability of bilinear quantum systems on graphs was preliminarily 
addressed by the second author in \cite{mio3,mio4}. There, the 
bilinear Schr\"odinger equation \eqref{eq1} is considered in the Hilbert space $L^2(\Gi,\C)$ with $\Gi$ a 
compact graph and $A$ a suitable 
self-adjoint Laplacian. One of the main difficulties of this framework is due to the nature of the spectrum of 
$A$. In particular, when we consider its ordered sequence of eigenvalues $(\lambda_k)_{k\in\N^*}$, it is possible 
to show that there exists $\MM\in\N^*$ such that
\begin{equation}\label{g13}\begin{split}
\inf_{{k\in\N^*}}|\lambda_{k+\MM}-\lambda_k|>0\\
\end{split}\end{equation}
(as ensured in \cite[Lemma\ 2.4]{mio3}). Nevertheless, the uniform {spectral gap} 
$\inf_{{k\in\N^*}}|\lambda_{k+1}-\lambda_k|> 0$ is only valid when $\Gi=(0,1)$. This hypothesis was crucial for 
the 
techniques adopted in the previous works on bounded intervals, which 
could not be applied in this framework. 
To this purpose, new spectral techniques were developed in the works \cite{mio3,mio4} in order to ensure the 
global exact controllability of the bilinear Schr\"odinger equation \eqref{eq1} on compact graphs.

\smallskip

When we consider the bilinear Schr\"odinger equation \eqref{eq1} on infinite graphs instead, a natural obstacle to 
the controllability is the loss of localization of the wave packets 
during the evolution: the dispersion. This effect can be measured by $L^\infty$-time decay, which implies a 
spreading out of the solutions, due to the time invariance of the $L^2$-norm. Dispersive estimations on 
infinite graphs can be found in \cite{AAN17,AAN15}. The other side of the same coin is that a self-adjoint 
Laplacian $A$ on $L^2(\Gi,\C)$ where $\Gi$ is an infinite graph, does not admit compact resolvent and then, the 
spectral techniques from \cite{mio3,mio4} can not be directly applied to this framework.

\smallskip

Despite the dispersive behavior of the bilinear Schr\"odinger equation \eqref{eq1} on infinite graphs, 
the authors addressed the problem in \cite{mio5} by exploiting 
a simple but still effective idea. When $\Gi$ contains suitable substructures, 
the 
Laplacian $A$ admits discrete spectrum corresponding to some 
specific eigenmodes. Such states are preserved by the 
dynamics of \eqref{eq1} for suitable choices of $B$ and they are not affected by the dispersive 
behavior of the equation. By working on 
the space spanned by such eigenmodes, global exact controllability results for the equation
\eqref{eq1} can be ensured in suitable subspaces of $L^2(\Gi,\C)$ with $\Gi$ an infinite graph, as presented in 
\cite{mio5}. We underline that the considered eigenmodes are
supported in compact 
sub-graphs of $\Gi$ and then, the result is only valid for suitable states vanishing on the infinite 
edges of the graph.

\smallskip

From this perspective, our purpose is natural. We aim to carry on the existing theory by proving the
controllability of \eqref{eq1} for quantum states that do not vanish on the infinite edges of the graph. 
In this regard, we consider the bilinear Schr\"odinger equation \eqref{eq1} for periodic functions. This choice 
allows us 
to have non-compactly supported eigenmodes and then, to ensure the exact 
controllability for states also 
defined on the infinite edges of the graph.

\medskip

{\noindent \bf \underline{Scheme of the work}}

\smallskip

The paper is organized as follows. In Section \ref{preli}, we introduce the main notations of the 
work. In Sections \ref{sec2} and Section \ref{sec3}, we respectively prove the global exact controllability when 
$\Gi$ is an infinite 
tadpole graph and an infinite star graph. In the last section, we generalize the previous results to 
some general infinite graphs. 

\section{Preliminaries}\label{preli}

Let $\Gi$ be a general graph composed by $N$ finite edges $\{e_j\}_{1\leq j\leq N}$ of lengths $\{L_j\}_{1\leq 
j\leq N}$ and $\widetilde N$ half-lines $\{e_j\}_{N+1\leq j\leq N+\widetilde N}$. Each edge $e_j$ with $j\leq N$ 
is 
associated to a coordinate starting from $0$ and going to $L_j$, while $e_j$ with $N+1\leq j\leq N+\widetilde N$ is 
parametrized with a coordinate starting from $0$ and going to $+\infty$. We consider $\Gi$ as domain of functions 
$$f:=(f^1,...,f^{N+\widetilde N}):\Gi\rightarrow \C, \ \ \ \ \ \ \ \ \ f^j:e_j\rightarrow \C,\ \ \ \ \ 1\leq j\leq 
N+\widetilde N.$$ Let $\{L_j\}_{ N+1\leq j\leq N+\widetilde N}\subset \R^+$. We consider the Hilbert space 
\begin{align}\label{spazio}L^2_p:=L^2_p(\Gi,\C)= \Big(\prod_{j=1}^{N}L^2(e_j,\C)\Big)\times 
\Big(\prod_{j=N+1}^{N+\widetilde 
N}L^2_p(e_j,\C)\Big),\ \ \ \ \ \text{with}\end{align}
$$L^2_p(e_j,\C)=\Big\{f\in L^2_{loc}(e_j,\C)\ :\ f(\cdot)=f\big(\ \cdot\ +2\pi kL_j\big),\ \ \ \ \forall 
k\in\N^*\Big\},\ \ \ N+1\leq j\leq N+\widetilde N.$$
The Hilbert spaces $L^2_p$ is equipped with the norm $\|\cdot\|_{L^2_p}$ induced by the scalar product
$$\la\psi,\ffi\ra_{L^2_p}=\sum_{j=1}^{N+\widetilde N}\int_{0}^{L_j}\overline{\psi^j}(x)\ffi^j(x)dx,\ \ \ \  \ \ 
\forall \psi,\ffi\in L^2_p.$$
We introduce the spaces $$H^s_p:= L^2_p\cap\Bigg(\Big(\prod_{j=1}^{N}H^s(e_j,\C)\Big)\times 
\Big(\prod_{j=N+1}^{N+\widetilde N}H^s_{loc}(e_j,\C)\Big)\Bigg)$$ with $s>0$. For $T>0$, we consider the bilinear 
Schr\"odinger equation in $L^2_p$
\begin{equation}\label{mainx1}\tag{BSE}\begin{split}
\begin{cases}
i\dd_t\psi(t)=-A\psi(t)+u(t)B\psi(t),\ \ \ \ \ \ \ \ &t\in(0,T),\\
\psi(0)=\psi_0.\\
\end{cases}
\end{split}
\end{equation}
The operator $A$ is a Laplacian equipped with suitable boundary conditions such that $D(A)\subseteq H^2_p$. The 
operator $B$ is a bounded symmetric operator in $L^2_p$ and $u\in L^2((0,T),\R)$ with $T>0$. 
We respectively denote
$$\upvarphi:=\{\ffi_k\}_{k\in\N^*},\ \ \ \ \  \ \ \ \ \ \  \ \upmu:=\{\mu_k\}_{k\in\N^*}$$ an orthonormal system 
of 
$L^2_p$ made by some eigenfunctions of $A$ and the corresponding eigenvalues. For $s>0$, we define the spaces
$$\Hi(\upvarphi):=\overline{span\{\varphi_k\ |\ k\in\N^*\}}^{\ L^2_p},$$
\begin{equation}\label{spaces}\begin{split}
&H^s_{\Gi}(\upvarphi):=\{\psi\in \Hi(\upvarphi)\ | \ \sum_{k\in\N^*}|k^s\la\ffi_k,\psi\ra_{L^2_p}|^2<\infty\},\\
&h^s:=\Big\{\{a_k\}_{k\in\N^*}\in\ell^2(\C)\ \big|\ \sum_{k\in\N^*}|k^{s}a_k|^2<\infty\Big\}.\\
\end{split}\end{equation}
We respectively equip $H^s_{\Gi}(\upvarphi)$ and $h^s$ with the norms $\|\cdot\|_{(s)}=\big({\sum_{k \in 
\N^*}}|k^s\la\ffi_k,\cdot\ra_{L^2_p}|^2\big)^\frac{1}{2}$ and $$\|{\bf 
x}\|_{(s)}=\big(\sum_{k\in\N^*}|k^sx_k|^2\big)^\frac{1}{2},\ \ \ \ \ \ \ \  \ \forall {\bf x}:=(x_k)_{k\in\N^*}\in 
h^s.$$
\begin{oss}
The space $\Hi(\upvarphi)$ is usually strictly smaller than $L^2_p$. If for instance we consider $\Gi$ as a 
ring parametrized from $0$ to $1$ and $\upvarphi=\big\{\sqrt{2}\sin(2k\pi x)\big\}_{k\in\N^*}$, then 
$\Hi(\upvarphi)$ is composed by those $L^2_p$ states which are odd with respect to the point $x=1/2$ and clearly 
$\Hi(\upvarphi)\subset L_p^2.$ 
\end{oss}
\begin{oss}
	Let $\mu_k\sim k^2$ and $c\in \R^+$ be such that $0\not\in\sigma(A+c,\Hi(\upvarphi))$ (the spectrum of $A+c$ in 
the Hilbert space $\Hi(\upvarphi)$).  For every $s>0,$ there exist $C_1,C_2>0$ such that $$C_1\|\psi\|_{(s)}\leq 
\||A+c|^{s/2}\psi\|_{L^2_p}\leq C_2\|\psi\|_{(s)},\ \ \ \ \ \ \ \ \forall \psi \in H^s_\Gi(\upvarphi).$$
\end{oss}
Let $\G_T^u$ be the unitary propagator (when it is defined) corresponding to the dynamics of \eqref{mainx1} in the 
time interval $[0,T]$.
\begin{defi}\label{exact}
Let $\upvarphi$ be an orthonormal system 
of 
$L^2_p$ made by some eigenfunctions of $A$ and $s>0$.
The bilinear Schr\"odinger equation \eqref{mainx1} is said to be 
globally exactly controllable in 
$H^s_{\Gi}(\upvarphi)$ when, for every $\psi_1,\psi_2\in H^s_\Gi(\upvarphi)$ 
such that $\|\psi_1\|_{L^2_p}=\|\psi_2\|_{L^2_p}$, there exist $T>0$ and $u\in L^2((0,T),\R)$ such that
	$$\G_T^u\psi_1=\psi_2.$$	
\end{defi}

The aim of the work is to study the global exact controllability of the \eqref{mainx1} on infinite graphs in 
suitable spaces $H^s_{\Gi}(\upvarphi)$ with $s>0$.

\section{Infinite tadpole graph} \label{sec2}

Let $\Ti$ be an {\it infinite tadpole graph} composed by two edges $e_1$ and $e_2$. The self-closing edge $e_1$, 
the \virgolette{head}, is connected to $e_2$ in the vertex $v$ and it is parametrized in the clockwise direction 
with a coordinate going from $0$ to $1$ (the length of $e_1$). The \virgolette{tail} $e_2$ is an half-line equipped 
with a coordinate starting from $0$ in $v$ and going to $+\infty$. The tadpole graph presents a natural symmetry 
axis that we denote by $r$.
\begin{figure}[H]
	\centering
	\includegraphics[width=\textwidth-100pt]{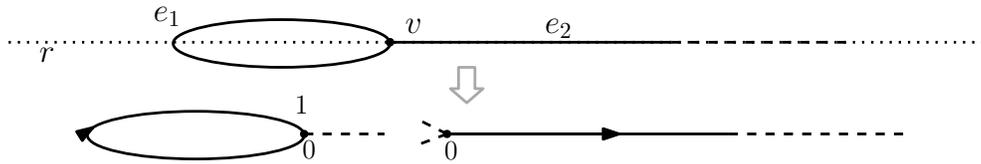}
	\caption{The parametrization of the infinite tadpole graph and its natural symmetry axis 
$r$.}\label{parametrizzazionetadpole1}
\end{figure}
Let $L^2_p$ be composed by functions which are periodic on the tail with period $1$, {\it i.e.} $L_2=1$. We 
consider the bilinear Schr\"odinger equation \eqref{mainx1} in $L^2_p$ with $A=-\Delta$ the Laplacian equipped with 
{\it Neumann-Kirchhoff\,} boundary conditions in the vertex $v$, {\it i.e.}
\begin{equation*}\begin{split}
D(A)=\Big\{\psi=(\psi^1,\psi^{2})\in H^2_p\ : \ \psi^1(0)=\psi^1(1)=\psi^2(0),\ \ \ \ \frac{\dd 
\psi^1}{\dd x}(0)-\frac{\dd 
\psi^1}{\dd x}(1)+\frac{\dd \psi^2}{\dd x}(0)=0\Big\}.\\
\end{split}
\end{equation*}
\begin{osss}\label{peculiarity}
The chosen operator $A$ is not self-adjoint in the Hilbert space $L^2_p$. This fact is an important 
peculiarity of this work with respect to the existing ones on bilinear quantum systems. However, we show how 
to 
construct subspaces of $L^2_p$ composed by eigenspaces of $A$ where the well-posedness and the controllability can 
be ensured.
\end{osss}

We assume the control field $B:\psi=(\psi^1,\psi^2)\longmapsto (V^1\psi^1,V^2\psi^2)$ being such that 
$$V^1(x)=x^2(x-1)^2,\ \ \ \ \ \ \ \ \ \ V^2(x)=\sum_{n\in\N}(x-n)^2(x-n-1)^2\upchi_{[n,n+1]}(x).$$
The choice of the potentials $V_1$ and $V_2$ is calibrated so that $B$ preserves 
the space $L^2_p$ and $V^1\psi^1\equiv V^2\psi^2|_{[n,n+1]}$ for every $n\in\N$ when
$\psi=(\psi^1,\psi^2)\in L^2_p$ is such that $\psi^1\equiv\psi^2|_{[n,n+1]}$ for every $n\in\N$.
In this framework, the \eqref{mainx1} 
corresponds to the two following Cauchy systems respectively in $L^2(e_1,\C)$ and $L^2_p(e_2,\C)$
\begin{equation}\tag{BSEt}\label{mainT}\begin{split}
\begin{cases}
i\dd_t\psi^1=-\Delta\psi^1+uV^1\psi^1,\\
\psi^1(0)=\psi^1_0,\\
\end{cases}\ \ \ \  \ \begin{cases}
i\dd_t\psi^2=-\Delta\psi^2+uV^2\psi^2,\\
\psi^2(0)=\psi^2_0.\\
\end{cases}
\end{split}
 \end{equation}
Let $\upvarphi:=\{\ffi_k\}_{k\in\N^*}$ be an orthonormal system of $L^2_p$ made by eigenfunctions of $-\Delta$ and 
corresponding to the eigenvalues $\upmu:=\{\mu_k\}_{k\in\N^*}$ such that, for every $k\in \N^*\setminus\{1\}$,
\begin{equation*}\begin{split}\begin{cases}
\ffi_k=\big(\cos({2(k-1)}\pi x),\cos({2(k-1)}\pi x)\big),\ \ \ \ \ \ \ \ \ \ &\mu_k={4(k-1)^2\pi^2},\\
\ffi_1=\big(\frac{\sqrt{2}}{2},\frac{\sqrt{2}}{2}\big),\ \ \ \ \ \ \ \ \ \ &\mu_1=0.
                              \end{cases}\end{split}\end{equation*}
	
\begin{osss}\label{equivalentdefi} We notice that each $f=(f^1,f^2)\in L^2_p$ belongs to $\Hi(\upvarphi)$ when
\begin{itemize}
 \item $f^1$ is symmetric with respect to the symmetry axis $r$ of $\Ti$;
 \item $f^2$ has period $2\pi$ and $f^2|_{[2n\pi,2(n+1)\pi]}\equiv f^1$ for every $n\in\N$. 
 \end{itemize}
 \end{osss}

\begin{prop}\label{lauraT}
	Let $\psi_0\in H^{4}_{\Ti}(\upvarphi)$ and $u\in L^2((0,T),\R)$. There
	exists a unique mild solution of the (\ref{mainT}) in
	$H^{4}_{\Ti}(\upvarphi)$, i.e. a function $\psi\in C^0\big([0,T],H^{4}_{\Ti}(\upvarphi)\big)$ such that
	\begin{equation}\label{mild}
	\psi(t)=e^{i\Delta t}\psi_0-
	i\int_0^t e^{i\Delta(t-s)}u(s)B\psi(s)ds .\\
	\end{equation}
	Moreover, the flow of \eqref{mainT} on $\Hi(\upvarphi)$ can be extended to a unitary flow $\G_t^{u}$ with 
respect to the $L^2_p-$norm such 
that $\G_t^{u}\psi_0=\psi(t)$ for any solution $\psi$ of \eqref{mainT} with initial data $\psi_0\in 
\Hi(\upvarphi)$.
\end{prop}
\begin{proof} 
{\bf 1) Unitary flow.} We consider Remark \ref{equivalentdefi}. For every $f=(f^1,f^2)\in \Hi(\upvarphi)$, we 
notice that $(B f)^1$ inherits from $f^1$ the property of being symmetric with respect to the symmetry axis $r$, 
while $(Bf)^2|_{[2n\pi,2(n+1)\pi]}\equiv(B f)^1$ for every $n\in\N$ as $f^2|_{[2n\pi,2(n+1)\pi]}\equiv f^1$ for 
every $n\in\N$. Now, $(Bf)^2$ has period $2\pi$ and $(Bf)^2(x)=(Bf)(2(n+1)\pi-x)$ for every $n\in\N$ and 
$x\in[2n\pi,(2n+1)\pi].$ Thus, $Bf=(V^1f^1,V^2f^2)\in \Hi(\upvarphi)$ for every $f=(f^1,f^2)\in \Hi(\upvarphi)$ 
and the 
control field $B$ preserves $\Hi(\upvarphi)$. The space $\Hi(\upvarphi)$ is a Hilbert space where the operator $A$ 
is self-adjoint and $B$ is bounded 
symmetric. 
Thanks to \cite[Theorem\ 2.5]{ball}, the \eqref{mainT} admits a unique solution $\psi\in 
C^0([0,T],\Hi(\upvarphi))$ 
for every $T>0$ and $\psi_0\in\Hi(\upvarphi)$. 
The flow of \eqref{mainT} is unitary in $\Hi(\upvarphi)$ thanks to the following arguments. If $u\in 
C^0((0,T),\R)$, then $\psi\in C^1((0,T),\Hi(\upvarphi))$ and $\dd_t\|\psi(t)\|_{L^2_p}^2=0$ from (\ref{mainT}). 
Thus $\|\psi(t)\|_{L^2_p} =\|\psi_0\|_{L^2_p}$. The generalization for $u\in L^2((0,T),\R)$ follows from a 
classical density argument, which ensures that the flow of the dynamics of the \eqref{mainT} is unitary in 
$\Hi(\upvarphi)$.

\smallskip
\needspace{3\baselineskip}

	\noindent
	{\bf 2) Regularity of the integral term in the mild solution.} The remaining part of the proof refers to the 
techniques leading to \cite[Lemma\ 1;\ Proposition\ 2]{laurent} (also adopted in the proof of 
\cite[Proposition\ 
2.1]{mio5}). Let $\psi\in C^0([0,T],H_{\Ti}^4(\upvarphi))$ with $T>0$. We notice $B\psi(s) \in 
H_p^{4}\cap H_{\Ti}^2(\upvarphi)$ for 
almost every $s\in (0,t)$ and $t\in(0,T)$. Let $G(t,x)=\int_0^t e^{i\Delta(t-s)}u(s)B\psi(s,x)ds$ so that
	\begin{equation*}\begin{split}
	\|G(t)\|_{(4)}&=\Big(\sum_{k\in\N^*}\Big|k^4\int_0^t e^{i\mu_ks}\la\ffi_k,u(s)B\psi(s,\cdot)\ra_{L^2_p} 
ds\Big|^2\Big)^\frac{1}{2}.\\
	\end{split}\end{equation*}
	For $f(s,\cdot):=u(s)B\psi(s,\cdot)$ such that $f=(f^1,f^2)$ and $k\in\N^*\setminus\{1\},$ we have
	\begin{equation*}\begin{split}
	&\la\ffi_k,f(s)\ra_{L^2_p}=-\frac{1}{\mu_k 
}\Big(\int_{0}^1\ffi_k^1(y)\dd_{x}^2f^1(s,y)dy+\int_{0}^{1}\ffi_k^2(y)\dd_{x}^2f^2(s,y)dy\Big)\\
	&=-\frac{2}{\mu_k }\int_{0}^1\ffi_k^1(y)\dd_{x}^2f^1(s,y)dy=\frac{1}{4 (k-1)^3\pi^3 }\int_{0}^1\sin(2 (k-1)\pi 
x)\dd_{x}^3f^1(s,y)dy\\
	&=\frac{1}{8(k-1)^4\pi^4 }\Bigg(\dd_{x}^3f^1(s,1)-\dd_{x}^3f^1(s,0)-\int_{0}^1\cos(2 (k-1)\pi x) 
\dd_{x}^4f^1(s,y)dy\Bigg).\\
	\end{split}\end{equation*}
	In the last relations, we considered  
$\ffi_k^1(\cdot)\dd_{x}^2f^1(s,\cdot)|_{[0,1]}=\ffi_k^2(\cdot)\dd_{x}^2f^2(s,\cdot)|_{[0,1]}$ as 
$\dd_{x}^2f^1(s,\cdot)|_{[0,1]}=\dd_{x}^2f^2(s,\cdot)|_{[0,1]}$. Equivalently to the first point of the proof of 
\cite[Proposition\ 2.1]{mio5}, there exists $C_1>0$ such that
	\begin{equation*}\begin{split}
	\|G(t)\|_{(4)}\leq& C_1\Big(\Big\|\int_0^t \big(\dd_{x}^3f^1(s,1)-\dd_{x}^3f^1(s,0)\big)e^{i \mu_{(\cdot)}s}ds 
\Big\|_{\ell^2} +\sqrt{t}\|f\|_{L^2((0,t),H^4_p)}\Big).
	\end{split}\end{equation*}
Thanks \cite[Proposition\ B.6]{mio3}, there exists $C_2(t)>0$ uniformly bounded for $t$ in bounded intervals such 
that $\|G(t)\|_{(4)} \leq C_2(t)\|f(\cdot,\cdot)\|_{L^2((0,t),H^4_p)}.$ For every $t\in [0,T]$, the last 
inequality 
shows that $G(t)\in H^4_\Ti(\upvarphi)$ and the provided upper bound is uniform. The Dominated Convergence Theorem 
leads to 
$G\in C^0([0,T], H^4_{\Ti}(\upvarphi))$.

	\smallskip
	
	\needspace{3\baselineskip}

	\noindent
	{\bf 3) Conclusion. } As $Ran(B|_{H_{\Ti}^{4}}(\upvarphi))\subseteq H_p^{4}\cap H^2_{\Ti}(\upvarphi)\subseteq 
H^{4}_p$, we have $B\in 
\LL(H^{4}_{\Ti}(\upvarphi),H^{4}_p)$ thanks to the arguments of \cite[Remark\ 2.1]{mio1}. Let $\psi_0\in 
H^{4}_{\Ti}(\upvarphi)$. We consider the map $$F:\psi \in C^0([0,T],H^{4}_{\Ti}(\upvarphi))\mapsto\phi\in 
C^0([0,T],H^{4}_{\Ti}(\upvarphi)) ,$$
	$$\phi(t)=F(\psi)(t)=e^{i\Delta t}\psi_0-\int_0^te^{i\Delta(t-s)}u(s) B\psi(s)ds,\ \ \ \ \ \forall t\in 
[0,T].$$
	For every $\psi_1,\psi_2\in C^0([0,T], H^4_{\Ti}(\upvarphi))$, from the first point of the proof, there exists 
$C(t)>0$ 
uniformly bounded for $t$ lying on bounded intervals such that
	\begin{equation*}\begin{split}
	&\|F(\psi_1)-F(\psi_2)\|_{L^\infty((0,T),H^{4}_\Ti(\upvarphi))}\leq\left\|\int_0^{(\cdot)} 
e^{i\Delta((\cdot)-s)}u(s) 
B(\psi_1(s)-\psi_2(s))ds\right\|_{L^\infty((0,T),H^{4}_\Ti(\upvarphi))}\\
	&\leq C(T)\|u\|_{L^2((0,T),\R)}\iii B\iii_{\LL(H^{4}_{\Ti}(\upvarphi),H^{4}_p)} 
\|\psi_1-\psi_2\|_{L^\infty((0,T),H^{4}_\Ti(\upvarphi))}.\\
	\end{split}\end{equation*}
	If $\|u\|_{L^2((0,T),\R)}$ is small enough, then $F$ is a contraction and Banach Fixed Point Theorem yields 
the 
existence of $\psi \in C^0([0,T],H^{4}_{\Ti}(\upvarphi) )$ such that $F(\psi)=\psi.$ When $\|u\|_{L^2((0,T),\R)}$ 
is not sufficiently small, we decompose $(0,T)$ with a sufficiently thin partition $\{t_j\}_{0\leq j\leq n}$ with 
$n\in\N^*$ such that each $\|u\|_{L^2((t_{j-1},t_j),\R)}$ is so small such that $F$ defined on the interval 
$[t_{j-1},t_j]$ is a contraction. The well-posedness on $[0,T]$ is defined by gluing each flow defined in 
every interval of the partition. \qedhere
\end{proof}

We are finally ready to present the following global exact controllability result (Definition 
\ref{exact}).
\begin{teorema}\label{globalegirino}
	The (\ref{mainT}) is globally exactly controllable in $H^4_\Ti(\upvarphi)$.
	\end{teorema}
\begin{proof}
The statement is proved by using the arguments adopted in the proof of \cite[Theorem\ 2.2]{mio5}.

\smallskip

\noindent
	{\bf 1) Local exact controllability.} We notice that $\ffi_1(T)=e^{-i\mu_1T}\ffi_1=\ffi_1 $ with $T>0$ as the 
first eigenvalue $\mu_1$ is equal to $0$. For $\epsilon,T>0$, we define 
	$$O_{\epsilon}^{4}:=\big\{\psi\in H_{\Ti}^{4}(\upvarphi)\big|\ : \ \|\psi\|_{L^2_p}=1,\ \|\psi 
-\ffi_1\|_{(4)}<\epsilon\big\}.$$ We ensure there exist $T,\epsilon>0$ so that, for every $\psi\in 
O_{\epsilon}^{4}$, there exists $u\in L^2((0,T),\R)$ such that $\psi= \G^u_T\ffi_1.$
	The result can be proved by showing the surjectivity of the map $\G_T^{(\cdot)}\ffi_1:u\in 
L^2((0,T),\R)\longmapsto \psi \in O_{\epsilon}^{4}$ with $T>0$. Let 
$$\G_{t}^{(\cdot)}\ffi_1=\sum_{k\in\N^*}{\ffi_k(t)}\la \ffi_k(t),\G_{t}^{(\cdot)}\ffi_1\ra_{L^2_p}.$$ We recall the 
definition of $h^4$ provided in \eqref{spaces}. Let $\alpha$ be the map 
defined as the sequence with elements $\alpha_{k}(u)=\la \ffi_k(T), \G_{T}^u\ffi_1\ra_{L^2_p}$ for $k\in\N^*$ such 
that
	$$\alpha:L^2((0,T),\R)\longrightarrow Q:=\{{\bf x}:=\{x_k\}_{k\in\N^*}\in h^4(\C)\ |\ \|{\bf 
x}\|_{\ell^2}=1\}.$$ 
The local exact controllability follows from the local surjectivity of $\alpha$ in a neighborhood of 
$\alpha(0)=\updelta=\{\delta_{k,1}\}_{k\in\N^*}$ with respect to the $h^4-$norm. To this end, we consider the 
Generalized Inverse Function Theorem and we study the surjectivity of 
$\gamma(v):=(d_u\alpha(0))\cdot\ v$ the Fréchet derivative of $\alpha$. Let $B_{k,1}:=\la\ffi_k,B\ffi_1\ra_{L^2_p}$ 
with $k\in\N^*$. The map $\gamma$ is the sequence of elements $\gamma_{k}(v):=
	-i\int_{0}^Tv(\tau)e^{i(\mu_k-\mu_1)s}d\tau B_{k,1}$ with $k\in\N^*$ so that
	$$\gamma:L^2((0,T),\R)\longrightarrow T_{\updelta}Q=\{{\bf x}:=\{x_k\}_{k\in\N^*}\in h^4(\C)\ |\ 
ix_1\in\R\}.$$ 
	As $\mu_1=0$, the surjectivity of $\gamma$ corresponds to the solvability of the moments problem
	\begin{equation}\begin{split}\label{mome1}
	{x_{k}}{B_{k,1}^{-1}}=-i\int_{0}^Tu(\tau)e^{i\mu_k\tau}d\tau,\ \ \ \ \ \ \ \ \ \ \forall 
\{x_{k}\}_{k\in\N^*}\in T_{\updelta}Q\subset h^4. 
	\\
	\end{split}\end{equation} 
	By direct computation, there exists $C>0$ such that 
$|B_{k,1}|=|\la\ffi_k,B\ffi_1\ra_{L^2_p}|\geq\frac{C}{k^4}$ for every $k\in\N^*$ and
	$$\big\{x_k B_{k,1}^{-1}\big\}_{k\in\N^*}\in \ell^2,\ \ \ \ \ \ \ \ \ \ \ \ \ \ \ i{x_{1}}B_{k,1}^{-1}\in\R.$$ 
	In conclusion, the solvability of $(\ref{mome1})$ is guaranteed by \cite[Proposition\ B.5]{mio3} since 
$$\{ix_k 
B_{k,1}^{-1}\}_{k\in\N^*}\in\{\{c_k\}_{k\in\N^*}\in \ell^2\ |\ c_1\in\R\},\ \ \ \ \ \ 
\inf_{k\in\N^*}|\mu_{k+1}-\mu_k|={4\pi^2}.$$

	\smallskip
	
	\noindent
	{\bf 2) Global exact controllability.}  Let $T,\epsilon>0$ be so that {\bf 1)} is valid. Thanks to Remark 
\ref{approxT}, for any $\psi_1,\psi_2\in H^{4}_{\Ti}(\upvarphi)$ such that 
$\|\psi_1\|_{L^2_p}=\|\psi_2\|_{L^2_p}=p$, there exist $T_1,T_2>0$, $u_1\in L^2((0,T_1),\R)$ and $u_2\in 
L^2((0,T_2),\R)$ such that $$\|\G^{u_1}_{T_1}p^{-1}\psi_1-\ffi_1\|_{(4)}<{\epsilon},\ \ \ \ \ 
\|\G^{u_2}_{T_2}p^{-1}\psi_2-\ffi_1\|_{(4)}<{\epsilon}$$
	and $ p^{-1}\G^{u_1}_{T_1}\psi_1,p^{-1}\G^{u_2}_{T_2}\psi_2\in O_{\epsilon}^{4}.$ From {\bf 1)}, there exist 
$u_3,u_4\in L^2((0,T),\R)$ such that $$\G_T^{u_3}\G^{u_1}_{T_1}\psi_1=\G_T^{u_4}\G^{u_2}_{T_2}\psi_2=p\ffi_1\ \ \ 
\ 
 \Longrightarrow \ \ \ \ \ \exists T>0,\ \widetilde u\in L^2((0,\widetilde T),\R)\ \ :\ \ \G_{\widetilde 
T}^{\widetilde u}\psi_1=\psi_2. 
$$ \end{proof}


\smallskip

\noindent
Let $\Phi:=\{\phi_k\}_{k\in\N^*}$ be an orthonormal system of $L^2_p$ made by eigenfunctions of $-\Delta$ and 
corresponding to the eigenvalues $\Lambda:=\{\lambda_k\}_{k\in\N^*}$ such that
$$\phi_k=\big(\sqrt{{2}}\sin({2k}\pi x),0\big),\ \ \ \ \ \ \ \ \ \ \lambda_k={4k^2\pi^2},\ \ \ \ \ \ \ \ \ \ 
\forall k\in \N^*.$$ 
We notice that the results \cite[Theorem\ 2.1;\ Theorem\ 2.2]{mio5} are still valid in the current framework and 
they lead to the following proposition.

\begin{prop}\label{lauraT1}
	Let (\ref{mainT}) be considered with $V_1(x)=x(1-x)$ and $V_2=0$ The (\ref{mainT}) is well-posed and globally 
exactly controllable in 
$H^3_\Ti(\Phi)$.
\end{prop}

The techniques leading to Proposition \ref{lauraT}, Theorem \ref{globalegirino} and Proposition \ref{lauraT1} also 
imply the following corollary.

\begin{coro}\label{corogirino}
	Let (\ref{mainT}) be considered with $$V^1(x)=x(1-x)+x^2(x-1)^2,\ \ \ 
V^2(x)=\sum_{n\in\N}(x-n)^2(x-n-1)^2\upchi_{[n,n+1]}(x).$$ The (\ref{mainT}) is well-posed and globally exactly 
controllable in 
$H^4_\Ti(\upvarphi)$ and $H^3_\Ti(\Phi)$.
\end{coro}

\begin{osss}\label{extension}
The choice of the lengths $L_1=1$ and $L_2=1$ has been done in order to simplify 
the theory of the current section. Nevertheless, it is possible to obtain similar results by considering different 
parameters $L_1$ and $L_2$ such that $L_1/L_2\in\Q$. A very similar situation is considered in the next 
section for a star graph with infinite spokes.
\end{osss}

\section{Star graph with infinite spokes} \label{sec3}
Let $\Si$ be a {\it star graph} composed by $N$ segments $\{e_j\}_{1\leq j\leq N}$ of lengths $\{L_j\}_{1\leq 
j 
\leq N}$ and $\widetilde N$ half-lines $\{e_j\}_{N+1\leq j\leq N+\widetilde N}$. The edges are connected in the 
internal vertex $v$, while $\{v_j\}_{1\leq j\leq N}$ are the external vertices of $\Si$ (those vertices of 
$\Si$ connected with only one edge). Each $e_j$ with $1\leq j\leq N$ is associated to a coordinate starting from 
$0$ in $v_j$ and going to $L_j$, while $e_j$ with $N+1\leq j\leq N+\widetilde N$ is parametrized with a coordinate 
starting from $0$ in $v$ and going to infinite.
\begin{figure}[H]
	\centering
	\includegraphics[width=\textwidth-100pt]{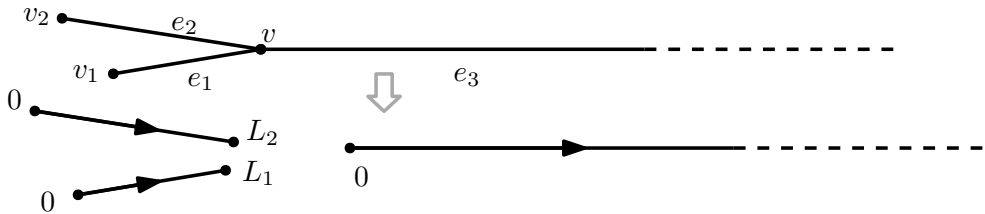}
	\caption{The parametrization of a star graph composed by $N=2$ segments and $\widetilde N=1$ half-lines.}
\end{figure}
Let $L^2_p$ be defined in \eqref{spazio}. This space is composed by functions which are periodic on the infinite 
edges with periods $\{L_j\}_{N+1\leq j\leq N+\widetilde N}$. We consider the bilinear Schr\"odinger equation 
\eqref{mainx1} in $L^2_p$ and the Laplacian $A=-\Delta$ being equipped with {\it Neumann-Kirchhoff\,} boundary 
conditions in $v$ and {\it Neumann\,}  boundary conditions in $\{v_{j}\}_{1\leq j\leq N}$, {\it i.e.}
\begin{equation*}\begin{split}
D(A)=&\Big\{\psi=(\psi^1,...,\psi^{N+\widetilde N})\in  H^2_p\ \  :\ \  \sum_{j=1}^{N}\frac{\dd \psi^j}{\dd 
x}(L_j)=\sum_{j=N+1}^{N+\widetilde N}\frac{\dd \psi^j}{\dd x}(0),\\
&\ \ \ \psi\in C^0(\Si,\C),\ \ \ \ \ \ \frac{\dd \psi^j}{\dd x}(v_j)=0\ \ \ \ \ \forall 1\leq j\leq N\Big\}.\\
\end{split}
\end{equation*}
Let $B:\psi\in L^2_p\mapsto B\psi=\big((B\psi)^1,...,(B\psi)^{N+\widetilde N}\big)$ be a bounded symmetric 
operator. The \eqref{mainx1} corresponds to the following Cauchy systems in $L^2(e_j,\C)$ when $1\leq j\leq N$ and 
in $L^2_p(e_j,\C)$ when $ N+1\leq j\leq N+\widetilde N$
\begin{equation}\tag{BSEs}\label{mainS}\begin{split}
\begin{cases}
i\dd_t\psi^j(t)=-\Delta\psi^j(t)+u(t)(B\psi)^j(t),& \ \ \ \ \ \ t\in(0,T),\\
\psi^j(0)=\psi^j_0.\\
\end{cases}
\end{split}
\end{equation}
\begin{osss}\label{peculiarity1}
As in Section \ref{sec2}, the chosen operator $A$ is not self-adjoint in the Hilbert space $L^2_p$. 
The central point here is to seek for the correct framework where the existence of eigenfunctions for 
$A$ is guaranteed. It is clear that the periodicity conditions on each infinite edge $e_j$ with $N+1\leq j\leq 
N+\widetilde N$ force any eigenvalue 
$\lambda$ of $A$ to be of the form $\frac{4k^2\pi^2}{L_j}$ with $k\in\N$. 
Thus, the eigenvalues $\lambda$ has to be contained in $\bigcap_{ N+1\leq j\leq N+\widetilde 
N}\big\{\frac{4k^2\pi^2}{L_j}\big\}_{k\in\N^*}$ which has to be non-empty. This is possible for suitable 
resonant lengths for the edges of the graphs. In the following part of this section we introduce a set of 
assumptions ensuring this fact.
\end{osss}
Let $L_{N+1}/L_j\in\Q$ for every $N+2\leq j\leq N+\widetilde N$. We denote by $l_j\in\N^*$ the smallest natural 
number such that 
\begin{equation}\label{multipliers}l_j\frac{L_{N+1}}{L_{j}}\in\N^*,\ \ \ \ \ \  \ \text{with}\ \ \ \ 1\leq j\leq 
N+\widetilde N.\end{equation} 
Let $n_k:=(k-1)\prod_{j=N+1}^{N+\widetilde N}l_j\frac{L_{N+1}}{L_j}\in\N^*$ for every $k\in\N^*$. We notice 
$$\bigcap_{j=N+1 }^{N+\widetilde N}\Big\{\frac{2 m\pi}{L_j}\Big\}_{m\in\N^*}= 
\Big\{\frac{2n_k\pi}{L_{N+1}}\Big\}_{k\in \N^*}.$$

\begin{assumptionA}
The numbers $\{L_j\}_{1\leq j\leq N+\widetilde N}$ are such that every ratio $\frac{L_{N+1}}{L_j}\in\Q$ for any 
$N+2\leq j\leq N+\widetilde N$. In addition, there exist $J\subseteq\N^*$ with $|J|=+\infty$ and $\{c_j\}_{N+1\leq 
j\leq N+\widetilde N}$ with $c_j\in[0,L_j]$ for any ${N+1\leq j\leq N+\widetilde N}$ such that
$$\sum_{j=1}^{N}\tan\Big(\frac{2 n_k \pi }{L_{N+1}}L_j\Big)=\sum_{j=N+1}^{N+\widetilde N}\tan\Big(\frac{2 n_k \pi 
}{L_{N+1}} c_j\Big),\ \ \ \  \ \ \ \ \ \forall k\in J.$$
In conclusion, the sequence $(\mu_k)_{k\in\N^*}=\Big(\frac{4 n_k^2 \pi^2 
}{L_{N+1}^2}\Big)_{k\in J}$ is such that $\mu_k\sim k^2$, {\it i.e.} there exist $C_1,C_2>0$ such that 
$$C_1k^2\leq \mu_k\leq C_2 k^2,\ \ \ \ \ \ \forall k\in\N^*.$$
\end{assumptionA}

When Assumptions A are satisfied, we define $\{\ffi_k\}_{k\in \N^*}$ such that
\begin{equation}\label{eigen}\begin{split}\begin{cases}
\ffi_1^j=\frac{1}{\sqrt{{(N+\widetilde N)L_j}}},\ \ \ \ \ &\forall\ 1\leq j\leq N+\widetilde N,\\
\ffi_k^1=\alpha_k\cos(\sqrt\mu_k x),\\
\ffi_k^j=\alpha_k\frac{\cos(\sqrt\mu_k L_1)}{\cos(\sqrt\mu_k L_j)}\cos(\sqrt\mu_k x),\ \ \ \  \ \ \ \ \ \ \ \ 
&\forall\ 2\leq j\leq N,\\
\ffi_k^j=\alpha_k\frac{\cos(\sqrt\mu_k L_1)}{\cos(\sqrt\mu_k c_j)}\cos(\sqrt\mu_k (x+c_j)\Big),\ \ \ \ \ \  \ \ \ 
\ 
\ \ \ \ &\forall\ N+1\leq j\leq N+\widetilde N\\
\end{cases}\end{split}\end{equation}
with $\alpha_k\in\C$ such that $\|\ffi_k\|_{L^2_p}=1$ and for every $k\in \N^*\setminus\{1\}$.

\begin{lemma}
Let $\Si$ be a star graph satisfying Assumptions A. The sequence $\{\ffi_k\}_{k\in \N^*}$ is an orthonormal system 
of $L_p^2$ made by eigenfunctions of the Laplacian $A$ corresponding to the eigenvalues $(\mu_k)_{k\in\N^*}$.
\end{lemma}
\begin{proof}
We notice that any eigenfunction $f=(f^1,...,f^{N+\widetilde N})$ of $A$ corresponding to an eigenvalue $\mu$ has 
to be such that $f^j$ has period $\frac{2\pi}{L_j}$ for every $N+1\leq j\leq N+\widetilde N$. Thus, 
$$\sqrt\mu\in\bigcap_{j=N+1 }^{N+\widetilde N}\Big\{\frac{2 m\pi}{L_j}\Big\}_{m\in\N^*}\supseteq 
\Big\{\frac{2n_k\pi}{L_{N+1}}\Big\}_{k\in J}.$$
Thanks to the Neumann boundary conditions in $\{v_j\}_{j\leq N}$ and to the periodicity conditions in 
$\{e_j\}_{N+1\leq j\leq N+\widetilde N}$, there exist $c_j\in[0,L_j]$ for any ${N+1\leq j\leq N+\widetilde N}$ 
such that 
\begin{equation*}\begin{split}\begin{cases}
f^j=\alpha_j\cos(\sqrt\mu x),\ \ \ \  \ \ \ \ &1\leq j\leq N,\\   
f^j=\alpha_j\cos(\sqrt\mu (x+c_j))+\beta_j\sin(\sqrt\mu (x+c_j)),\ \ \ \ \  \ \ \ \ &N+1\leq j\leq N+\widetilde 
N,\\
\end{cases}\end{split}\end{equation*} 
with suitable $\{\alpha_j\}_{j\neq N+\widetilde N},\{\beta_j\}_{N+1\leq j\neq N+\widetilde N}\subset\C$. 
The Neumann-Kirchhoff boundary conditions in $v$ yield
\begin{equation*}\begin{split}\begin{cases}\alpha_1\cos(\sqrt\mu L_1)=\alpha_j\cos(\sqrt\mu L_j),\ \ \ \ \ \ \ \ \ 
&\forall 2\leq j\leq N,\\
\alpha_1\cos(\sqrt\mu L_1)=\alpha_j\cos(c_j)+\beta_j\sin(c_j),\ \ \ \ \ \ \ \ \ &\forall N+1\leq j\leq 
N+\widetilde 
N.\\\end{cases}\end{split}\end{equation*}
When $\beta_j=0$ for every $N+1\leq j\leq N+\widetilde N$, the last identities implies
\begin{equation*}\begin{split}\begin{cases}
f^1=\alpha_j\cos(\sqrt\mu x),\\
f^j=\alpha_j\frac{\cos(\sqrt\mu L_1)}{\cos(\sqrt\mu L_j)}\cos(\sqrt\mu x),\ \ \ \  \ \ \ \ \ \ \ \ &\forall 2\leq 
j\leq N,\\
f^j=\alpha_j\frac{\cos(\sqrt\mu L_1)}{\cos(\sqrt\mu c_j)}\cos(\sqrt\mu (x+c_j)),\ \ \ \  \ \ \ \ \ \ \ \ &\forall 
2\leq j\leq N.\\
\end{cases}\end{split}\end{equation*}
We recall that the numbers $c_j\in[0,L_j]$ for every ${N+1\leq j\leq N+\widetilde N}$ are such that
$$\sum_{j=1}^{N}\tan(\sqrt\mu L_j)=\sum_{j=N+1}^{N+\widetilde N}\tan(\sqrt\mu c_j),\ \ \ \ \ \ \ \forall 
\sqrt\mu\in \Big\{\frac{2n_k\pi}{L_{N+1}}\Big\}_{k\in J}.$$
Thus, the second condition characterizing the Neumann-Kirchhoff boundary conditions is verified when $\beta_j=0$. 
As a 
consequence, $\{\ffi_k\}_{k\in \N^*}$ is composed by eigenfunctions of $A$. The 
orthonormality follows from the fact that $\big\{\cos\big(\frac{2\pi k}{L}x\big)\big\}_{k\in\N^*}$ is an 
orthogonal family in $L^2([0,L],\C)$ with $L>0$.
\qedhere
\end{proof}

Equivalently to Proposition \ref{lauraT}, we have the following well-posedness result.

\begin{prop}\label{lauraS}
Let the star graph $\Si$ satisfy Assumptions A. Let $B$ be a bounded symmetric operator in $L^2_p$ such that
$$B:\Hi(\upvarphi)\longrightarrow \Hi(\upvarphi),\ \ \ \  \ \ \  \ B:H^2_{\Si}(\upvarphi)\longrightarrow 
H^2_{\Si}(\upvarphi),\ \ \ \  \ \ \  \ B:H_{\Si}^{3}(\upvarphi)\longrightarrow H_p^{3}\cap H^2_{\Si}(\upvarphi).$$
	Let $\psi_0\in H^{3}_{\Si}(\upvarphi)$ and $u\in L^2((0,T),\R)$. There
	exists a unique mild solution $\psi\in C^0([0,T],H^3_{\Si}(\upvarphi))$ of (\ref{mainS}) with 
initial data $\psi_0$. The flow of \eqref{mainS} 
on 
$\Hi(\upvarphi)$ can be extended to a unitary flow $\G_t^{u}$ with respect to the $L^2_p-$norm such 
that $\G_t^{u}\psi_0=\psi(t)$ for any solution $\psi$ of \eqref{mainS} with initial data $\psi_0\in 
\Hi(\upvarphi)$.
\end{prop}
\begin{proof} 
The proof follows from the same arguments adopted in Proposition \ref{lauraT}. First, we notice that $A$ is 
self-adjoint in $\Hi(\upvarphi)$ and $B$ is bounded symmetric since $B:\Hi(\upvarphi)\rightarrow\Hi(\upvarphi)$. 
Second, we can define an unitary flow for the dynamics of the equation in $\Hi(\upvarphi)$ as in the proof of the 
mentioned proposition.

\smallskip
\needspace{3\baselineskip}

	\noindent
	{\bf 1) Regularity of the integral term in the mild solution.} Let $\psi\in C^0([0,T],H_{\Si}^3(\upvarphi))$ 
with $T>0$. 
We notice $B\psi(s) \in H^{3}_p\cap H_{\Si}^2(\upvarphi)$ for almost every $s\in (0,t)$ and $t\in(0,T)$. Let 
$G(t)=\int_0^t 
e^{i\Delta(t-s)}u(s)B\psi(s,x)ds$ so that
	\begin{equation*}\begin{split}
	\|G(t)\|_{(3)}&=\Big(\sum_{k\in\N^*}\Big|k^3\int_0^t e^{i\mu_ks}\la\ffi_k,u(s)B\psi(s,\cdot)\ra_{L^2_p} 
ds\Big|^2\Big)^\frac{1}{2}.\\
	\end{split}\end{equation*}
	Let $f(s,\cdot):=u(s)B\psi(s,\cdot)$. We define $\dd_x f(s)=(\dd_x f^1(s),...,\dd_x f^N(s))$ the derivative of 
$f(s)$. Thanks to the validity of 
Assumptions A, we have $\sqrt\mu_k\sim k$ and there exists $C_1>0$ such that, for every 
$k\in\N^*\setminus\{1\},$
	\begin{equation*}\begin{split}
	&\left|k^3\int_0^t e^{i\mu_ks}\la\ffi_k,f(s)\ra_{L^2_p} ds\right|
	\leq\frac{C_1}{k}\sum_{j=1}^{N+\widetilde N}\left(
	\left|\dd_{x}\ffi_k^j(L_j)\int_0^t e^{i\mu_ks}\dd_{x}^2f^j(s,L_j)ds\right|\right.\\
	&+\left.\left|\dd_{x}\ffi_k^j(0)\int_0^t e^{i\mu_ks}\dd_{x}^2f^j(s,0)ds\right|+ \left|\int_0^t 
e^{i\mu_ks}\int_{0}^{L_j}\dd_{x}\ffi_k^j(y)\dd_{x}^3f^j(s,y)dyds\right|\right).\\
	\end{split}\end{equation*}
	The argument of \cite[Remark\ 3.4]{mio5} yields that 
$\dd_{x}^3f(s,\cdot)\in\overline{span\big\{\mu_k^{-1/2}\dd_{x}\ffi_k:\ k\in\N^*\big\}}^{ L^2}$ for almost every 
$s\in (0,t)$ and $t\in(0,T)$, and there exists $C_2>0$ such that
	\begin{equation*}\begin{split}
	\|G(t)\|_{(3)}\leq &C_2\sum_{j=1}^{N+\widetilde N}\Big(\Big\|\int_0^t\dd_{x}^2 f^j(s,0)e^{i \mu_{(\cdot)}s}ds 
\Big\|_{\ell^2}+\Big\|\int_0^t\dd_{x}^2 f^j(s,L_j)e^{i \mu_{(\cdot)}s}ds \Big\|_{\ell^2}\Big)\\
	&+C_2\Big\|\int_0^t\big\la {\mu_{(\cdot)}^{-1/2}}\dd_x\ffi_{(\cdot)}(s),\dd_{x}^3 f(s)\big\ra_{L^2_p} e^{i 
\mu_{(\cdot)}s}ds\Big\|_{\ell^2}.\\
	\end{split}\end{equation*}
	From \cite[Proposition\ B.6]{mio3}, there exists $C_3(t)>0$ uniformly bounded for $t$ in bounded intervals 
such 
that $\|G\|_{(3)}\leq C_3(t)\|f(\cdot,\cdot)\|_{L^2((0,t),H^3_p)}.$ The provided upper bounds are uniform and the 
Dominated Convergence Theorem leads to $G\in C^0([0,T], H^3_{\Si}(\upvarphi))$.

	\smallskip
	\noindent
	{\bf 2) Conclusion.} We proceed as in the second point of the proof of Proposition \ref{lauraT}. Let 
$\psi_0\in 
H^{3}_{\Si}(\upvarphi)$. We consider the map $F:\psi \in C^0([0,T],H^{3}_{\Si}(\upvarphi))\mapsto\phi\in 
C^0([0,T],H^{3}_{\Si}(\upvarphi)) $ with
	$$\phi(t)=F(\psi)(t)=e^{i\Delta t}\psi_0-\int_0^te^{i\Delta(t-s)}u(s) B\psi(s)ds,\ \ \ \ \ \forall t\in 
[0,T].$$
	Let $L^\infty(H^{3}_\Si(\upvarphi)):=L^\infty((0,T),H^{3}_\Si(\upvarphi))$. For every $\psi_1,\psi_2\in 
C^0([0,T],H^{3}_{\Si}(\upvarphi))$, thanks to {\bf 1)}, there exists $C(t)>0$ uniformly bounded for $t$ lying on 
bounded intervals such that
	\begin{equation*}\begin{split}
	&\|F(\psi_1)-F(\psi_2)\|_{L^\infty(H^{3}_\Si(\upvarphi))}\leq C(T)\|u\|_{L^2((0,T),\R)}\iii 
B\iii_{\LL(H^{3}_{\Si}(\upvarphi),H_p^{3})} \|\psi_1-\psi_2\|_{L^\infty(H^{3}_\Si(\upvarphi))}.\\
	\end{split}\end{equation*}
	The Banach Fixed Point Theorem leads to the claim as in the mentioned proof. \qedhere
\end{proof}

By keeping in mind the definition of global exact controllability provided in Definition \ref{exact}, we present 
the following result.

\begin{teorema}\label{globalestella}
Let the hypotheses of Proposition $\ref{lauraS}$ be satisfied. We also assume that
\begin{enumerate}
		\item there exists $C>0$ such that $|\la\ffi_k,B\ffi_1\ra_{L^2_p}|\geq\frac{C}{k^{3}}$ for every 
$k\in\N^*$;
		\item for every $(j,k),(l,m)\in \N^2$ such that $(j,k)\neq (l,m)$, $j<k$, $l<m$ and 
$\mu_j-\mu_k=\mu_j-\mu_m,$ it holds 
$$\la\ffi_j,B\ffi_j\ra_{L^2_p}-\la\ffi_k,B\ffi_k\ra_{L^2_p}-\la\ffi_l,B\ffi_l\ra_{L^2_p}+\la\ffi_m,B\ffi_m\ra_{
L^2_p}\neq 0.$$
	\end{enumerate}
	The (\ref{mainS}) is globally exactly controllable in $H^3_\Si(\upvarphi)$.
	\end{teorema}
\begin{proof}{\bf 1) Local exact controllability.}  
The statement follows as Theorem \ref{globalegirino}. First, for $\epsilon,T>0$, the local exact controllability 
in 
$O_{\epsilon,T}^{3}:=\big\{\psi\in H_{\Si}^{3}(\upvarphi)\big|\ \|\psi\|_{L^2_p}=1,\ \|\psi 
-\ffi_1(T)\|_{(3)}<\epsilon\big\}$ 
with $\ffi_1(T)=e^{-i\mu_1T}\ffi_1$ is ensured by proving the surjectivity of the map
	$$\gamma:L^2((0,T),\R)\longrightarrow T_{\updelta}Q=\{{\bf x}:=\{x_k\}_{k\in\N^*}\in h^3(\C)\ |\ ix_1\in\R\},$$
	the sequence of elements $\gamma_{k}(v):=
	-i\int_{0}^Tv(\tau)e^{i(\mu_k-\mu_1)s}d\tau B_{k,1}$ with $B_{k,1}:=\la\ffi_k,B\ffi_1\ra_{L^2_p}$ 
for $k\in\N^*$. 
The surjectivity of $\gamma$ corresponds to the solvability of the moments problem
	\begin{equation}\begin{split}\label{mome1bis}
	{x_{k}}B_{k,1}^{-1}=-i\int_{0}^Tu(\tau)e^{i(\mu_k-\mu_1)\tau}d\tau,\ \ \ \ \ \ \ \ \ \ \forall 
\{x_{k}\}_{k\in\N^*}\in T_{\updelta}Q\subset h^3. 
	\\
	\end{split}\end{equation} 
	As there exists $C>0$ such that $|\la\ffi_k,B\ffi_1\ra_{L^2_p}|\geq\frac{C}{k^3}$ for every $k\in\N^*$, we 
have 
$\big\{x_k B_{k,1}^{-1}\big\}_{k\in\N^*}\in \ell^2$ and $i{x_{1}}B_{k,1}^{-1}\in\R.$ The solvability of 
$(\ref{mome1bis})$ is guaranteed by \cite[Proposition\ B.5]{mio3} since 
	$$\inf_{k\in\N^*}|\mu_{k+1}-\mu_k|\geq {\pi^2}\min\{L_j^{-2}:\ {N+1\leq j\leq N+\widetilde N}\}>0.$$

	\smallskip
	
	\noindent
	{\bf 2) Global exact controllability.}  The global exact controllability in $H^3_\Si(\upvarphi)$ is ensured as 
in the second point of the proof of Theorem \ref{globalegirino} by considering Remark \ref{approxS} instead of 
Remark \ref{approxT}. \qedhere
\end{proof}
\begin{oss}
Let $\{L_j\}_{1\leq j\leq N+\widetilde N}$ be such that $\frac{L_{N+1}}{L_j}\in\Q\ $  for any $1\leq 
j\leq N+\widetilde N$. We notice that Assumptions A are satisfied with $c_j=0$ for every $N+1\leq j\leq 
N+\widetilde N$. Indeed, let $l_j$ be the numbers defined in \eqref{multipliers}. The sequence
$$(\mu_k)_{k\in\N^*}:=\Big\{\frac{4\widetilde n_k^2\pi^2}{L_{N+1}^2}\Big\}_{k\in\N^*}\ \ \ \ \ \  \text{with }\ \ 
\ \ \widetilde n_k:=(k-1)\prod_{j=1}^{N+\widetilde N}l_j\frac{L_{N+1}}{L_j}$$ is composed by eigenvalues. The 
corresponding eigenfunctions $(\ffi_k)_{k\in\N^*}$ are provided in \eqref{eigen}. In this 
framework, $$\mu_k\sim k^2,\ \ \ \ \  \ \ \tan(\mu_kL_j)=0\ \ \ \ \ \  \ \ \forall k\in\N^*,\ \ \ 1\leq j\leq N.$$ 
Thus, the validity of 
Assumptions A is ensured with $c_j=0$ for every $N+1\leq j\leq N+\widetilde N$.
\end{oss}

\begin{oss}
Let $\Si$ satisfy Assumptions A. We consider $\{L_j\}_{1\leq j\leq N+\widetilde N}$ being such that 
$\frac{L_{N+1}}{L_j}\in\Q\ $  for any $1\leq 
j\leq N+\widetilde N$ so that the previous remark is verified. Let $\widetilde 
B:\psi\longmapsto (V^1\psi^1,...,V^{N+\widetilde N}\psi^{N+\widetilde N})$ be such 
that
\begin{equation*}\begin{split}\begin{cases}V^j(x)=x^2(x-L_j)^2,\ \ \ \ \ &\forall 1\leq j\leq N,\\
V^j(x)=\sum_{n\in\N}(x-nL_j)^2(x-(n+1)L_j)^2\upchi_{[nL_j,(n+1)L_j]}(x), &\forall N+1\leq j\leq N+\widetilde 
N.\end{cases}\end{split}\end{equation*}
If we consider the operator $B$ on $L^2_p$ such that $B\psi=\sum_{j=1}^{+\infty}\ffi_j\la\ffi_j,\widetilde 
B\psi\ra_{L_p^2},$ then the corresponding \eqref{mainS} is well-posed and globally exactly controllable in the 
space $H^4_\Si(\upvarphi)$. The result is proved by using the techniques leading to Proposition 
\ref{lauraT}, Proposition \ref{lauraS}, Theorem \ref{globalegirino} and Theorem \ref{globalestella}. In the next 
section, we ensure in the same way the well-posedness and 
the global exact controllability in $H^s_\Gi$ for suitable $s\geq 3$ with 
abstract $\Gi$ and $B$. 
\end{oss}
\section{Generic graphs} \label{sec4}
In this section, we study the controllability of the \eqref{mainx1} for a general graph $\Gi$ made by $N$ finite 
edges $\{e_j\}_{1\leq j\leq N}$ of lengths $\{L_j\}_{1\leq j\leq N}$, $\widetilde N$ half-lines $\{e_j\}_{N+1\leq 
j\leq N+\widetilde N}$ and $M$ vertices $\{v_j\}_{1\leq j\leq M}$. For every vertex $v$, we denote 
$N(v):=\big\{l \in\{1,...,N\}\ |\ v\in e_l\big\}$. We respectively call $V_e$ and $V_i$ the external and the 
internal vertices of $\Gi$, {\it i.e.}
$$V_e:=\big\{v\in\{v_j\}_{1\leq  j\leq M}\ |\ \exists ! e\in\{e_j\}_{1\leq j\leq N}: v\in e\big\},\ \ \ \ \ 
V_i:=\{v_j\}_{1\leq  j\leq M}\setminus V_e.$$

We consider the bilinear Schr\"odinger equation \eqref{mainx1} in $L^2_p$ for a general graph $\Gi$. The Laplacian 
$A=-\Delta$ is equipped with {\it Dirichlet} or {\it Neumann} boundary conditions in the external vertices, and 
the internal vertices are equipped with {\it Neumann-Kirchhoff} boundary 
conditions. More precisely, a vertex $v\in V_i$ is said to be equipped with Neumann-Kirchhoff boundary 
conditions when every 
function $f=(f^1,...,f^N)\in D(A)$ is continuous in $v$ and
\begin{equation*}\begin{split}
	\sum_{l\in N(v)}\frac{\dd f^l}{\dd x}(v)=0,
	\end{split}
	\end{equation*}
when the derivatives are assumed to be taken in the directions away from the vertex. We 
respectively call ($\Di$), ($\NN$) and ($\NN\KK$) the {\it Dirichlet}, {\it Neumann} and  {\it Neumann-Kirchhoff}
boundary conditions characterizing $D(A)$.

\smallskip

We say that a vertex $v$ of $\Gi$ is equipped with one of the previous boundaries, when each $f\in D(A)$ satisfies 
it in $v$. We say that $\Gi$ is equipped with ($\Di$) (or ($\NN$)) when, for every $f\in D(A)$, the function $f$ 
satisfies ($\Di$) (or ($\NN$)) in every $v\in V_e$ and verifies ($\NN\KK$) in every $v\in V_i$. In addition, the 
graph $\Gi$ is said to be equipped with ($\Di$/$\NN$) when, for every $f\in D(A)$ and $v\in V_e$, the function $f$ 
satisfies 
($\Di$) or ($\NN$) in $v$, and $f$ verifies ($\NN\KK$) in every $v\in V_i$. 

\smallskip

Let $\upvarphi:=\{\ffi_k\}_{k\in\N^*}$ be an orthonormal system of $L^2_p$ made by some eigenfunctions of $A$ and 
let 
$\{\mu_{k}\}_{k\in\N^*}$ be the corresponding eigenvalues. Let $[r]$ be the entire part of $r\in\R$. We define 
$\Gi(\upvarphi)=\bigcup_{k\in\N^*}supp(\ffi_k)$ and we respectively denote by $V_e(\upvarphi)$ and 
$V_i(\upvarphi)$ the external and internal vertices of $\Gi(\upvarphi)$. For $s>0$, we introduce the 
space
\begin{equation*}
		\begin{split}
		H^s_{\NN\KK}(\upvarphi):=\Big\{&\psi\in \Hi(\upvarphi)\cap H^s_p\ |\ \dd_x^{2n}\psi\text{ is continuous in 
}v,\ 
 \forall n\in\N,\ n<\big[({s+1})/{2}\big],\ \forall v\in V_i;\\
& \sum_{j\in N(v)}\dd_{x}^{2n+1}\psi^j(v)=0,\ \forall n\in\N,\ n<\big[{s}/{2}\big],\  \forall v\in V_i\Big\}.\\
		\end{split}
\end{equation*}

\begin{osss}\label{subgraph}
We notice the following facts.
\begin{itemize}
 \item $\Gi(\upvarphi)$ is a finite or infinite sub-graph of $\Gi$ whose structure depends on the orthonormal 
family $\upvarphi$.
 \item The functions belonging to $\Hi(\upvarphi)$, $H^s_\Gi(\upvarphi)$ and $H^s_{\NN\KK}(\upvarphi)$ can be 
considered as functions with domain $\Gi(\upvarphi)$. 
 \item $\Gi(\upvarphi)$ shares some external and internal vertices with $\Gi$. Its new external vertices are 
$V_e(\upvarphi)\setminus V_e$. 
 \item Let $L^2_p(\Gi(\upvarphi),\C)$ be the 
space defined from the identities \eqref{spazio} by considering the graph 
$\Gi(\upvarphi)$. Each $\ffi_k|_{\Gi(\upvarphi)}$ is an eigenfunction of a Laplacian $\widetilde A$ defined on 
$L^2_p(\Gi(\upvarphi),\C)$ as follows. The domain $D(\widetilde A)$ is composed by the restriction in 
$\Gi(\upvarphi)$ of those $H^2_p$ functions satisfying $(\Di)$ in the vertices $V_e(\upvarphi)\setminus 
V_e$ and verifying the same boundary conditions defining $D(A)$ in the vertices 
$V_i(\upvarphi)\cup \big(V_e(\upvarphi)\cap V_e\big)$.
 \end{itemize}
 \end{osss}
From now on, when we claim that the vertices of $\Gi(\upvarphi)$ are equipped with any type of boundary 
conditions, 
this is done in the meaning of Remark \ref{subgraph}. Let $\eta>0,$ $a\geq 0$ and $$I:=\{(j,k)\in(\N^*)^2:j< 
k\}.$$		

\needspace{3\baselineskip}
\begin{assumptionI}[$\upvarphi,\eta$] Let $B$ be a bounded and symmetric operator in $L^2_p$ satisfying the 
following conditions.
	\begin{enumerate}
		\item There exists $C>0$ such that $|\la\ffi_k,B\ffi_1\ra_{L^2_p}|\geq\frac{C}{k^{2+\eta}}$ for every 
$k\in\N^*$.
		\item For every $(j,k),(l,m)\in I$ such that $(j,k)\neq(l,m)$ and $\mu_j-\mu_k=\mu_j-\mu_m,$ it holds 
$\la\ffi_j,B\ffi_j\ra_{L^2_p}-\la\ffi_k,B\ffi_k\ra_{L^2_p}-\la\ffi_j,B\ffi_j\ra_{L^2_p}+\la\ffi_m,B\ffi_m\ra_{L^2_p
}\neq 0.$
	\end{enumerate}
	
\end{assumptionI}

\begin{assumptionII}[$\upvarphi,\eta,a$] We have  $B:\Hi(\upvarphi)\rightarrow\Hi(\upvarphi)$ and 
$Ran(B|_{H^2_{\Gi}(\upvarphi)})\subseteq H^2_{\Gi}(\upvarphi).$ In addition, one of the following points is 
satisfied.
	\begin{enumerate}
		
		\item When $\Gi(\upvarphi)$ is equipped with ($\Di$/$\NN$) and $a+\eta\in(0, 3/2)$, there exists 
		$d\in[\max\{a+\eta,1\},3/2)$ such that $$Ran(B|_{H_{\Gi}^{2+d}(\upvarphi)})\subseteq H^{2+d}_p\cap 
H^2_{\Gi}(\upvarphi).$$

		\item When $\Gi(\upvarphi)$ is equipped with ($\NN$) and $a+\eta\in(0, 7/2)$, there exist 
$d\in[\max\{a+\eta,2\},7/2)$ and $d_1\in(d,7/2)$ such that $$Ran(B|_{ H^{d_1}_{\NN\KK}(\upvarphi)})\subseteq 
H^{d_1}_{\NN\KK}(\upvarphi),\ \ \ \ \ Ran(B|_{H_{\Gi}^{2+d}(\upvarphi)})\subseteq H^{2+d}_p\cap 
H^{1+d}_{\NN\KK}(\upvarphi)\cap H^2_{\Gi}(\upvarphi).$$

		\item When $\Gi$ is equipped with ($\Di$) and $a+\eta\in(0, 5/2)$, there exists 
$d\in[\max\{a+\eta,1\},5/2)$ such that $$Ran(B|_{H_{\Gi}^{2+d}(\upvarphi)})\subseteq H^{2+d}_p\cap 
H^{1+d}_{\NN\KK}(\upvarphi)\cap H^2_{\Gi}(\upvarphi).$$ If $d\geq 2$, then there exists $d_1\in(d,5/2)$ such 
that $Ran(B|_{H^{d_1}_p\cap \Hi(\upvarphi)})\subseteq H^{d_1}_p\cap \Hi(\upvarphi).$

	\end{enumerate}
\end{assumptionII}

From now on, we omit the terms $\upvarphi,$ $\eta$ and $a$ from the notations of Assumptions I and Assumptions II 
when their are not relevant.

We are finally ready to present some {\it interpolation properties} for the spaces $H^s_{\Gi}(\upvarphi)$ with 
$s>0$. 
\begin{prop}\label{bor}Let $\upvarphi:=\{\ffi_k\}_{k\in\N^*}$ be an orthonormal system of $L^2_p$ made by 
eigenfunctions of $A$.
	
	\smallskip
	\noindent
	{\bf 1)} If the graph $\Gi(\upvarphi)$ is equipped with ($\Di$/$\NN$), then
	$$H^{s_1+s_2}_{\Gi}(\upvarphi)=H_{\Gi}^{s_1}(\upvarphi)\cap H^{s_1+s_2}_p \ \ \ \text{for}\ \ \ s_1\in\N,\ 
s_2\in[0,1/2).$$
	
	\noindent
	{\bf 2)} If the graph $\Gi(\upvarphi)$ is equipped with ($\NN$), then
	$$H^{s_1+s_2}_{\Gi}(\upvarphi)=H_{\Gi}^{s_1}(\upvarphi)\cap H^{s_1+s_2}_{\NN\KK}(\upvarphi) \ \ \ \text{for}\ 
\ 
\ s_1\in 2\N\,\ s_2\in[0,3/2).$$
	
	\noindent
	{\bf 3)} If the graph $\Gi(\upvarphi)$ is equipped with ($\Di$), then
	$$H^{s_1+s_2+1}_{\Gi}(\upvarphi)=H_{\Gi}^{s_1+1}(\upvarphi)\cap H^{s_1+s_2+1}_{\NN\KK}(\upvarphi) \ \ \ 
\text{for}\ \ \ s_1\in 2\N,\ s_2\in[0,3/2).$$
\end{prop}
\begin{proof} Let us start by considering the first point of the statement. We denote by $\{e_j\}_{j\leq N_1}$ the 
finite edges composing $\Gi(\upvarphi)$, while
$\{e_j\}_{N_1+1\leq j\leq N_1+\widetilde N_1}$ are its infinite edges corresponding to the periods 
$\{L_j\}_{N_1+1\leq j\leq N_1+\widetilde N_1}$. We define a compact graph $\widetilde \Gi(\upvarphi)$ from 
$\Gi(\upvarphi)$ as follows (see Figure \ref{parametrizzazionetadpole2} for further details). For every $N_1+1\leq 
j\leq 
N_1+\widetilde N_1$, we cut the edge $e_j$ at distance $L_j$ from the internal vertex of $\Gi(\upvarphi)$ where 
$e_j$ is connected. As $\widetilde \Gi(\upvarphi)$ is a compact graph, the space
$L^2_p(\widetilde \Gi(\upvarphi),\C)$ corresponds to $L^2(\widetilde \Gi(\upvarphi),\C)$. There, we consider 
a 
self-adjoint Laplacian $\widetilde A$ being 
defined 
as follows. Every internal vertex of $\widetilde \Gi(\upvarphi)$ is equipped with Neumann-Kirchhoff boundary 
conditions. Every external vertex of $\widetilde \Gi(\upvarphi)$ belonging to $V_e(\upvarphi)$ is equipped with 
the 
same boundary conditions of $ \Gi(\upvarphi)$, while every other external vertex is equipped with 
($\Di$). 
Finally, we denote by $H^s_{\widetilde \Gi(\upvarphi)}:=D(|\widetilde A|^\frac{s}{2})$ for every $s>0$. 
\begin{figure}[H]
	\centering
	\includegraphics[width=\textwidth-100pt]{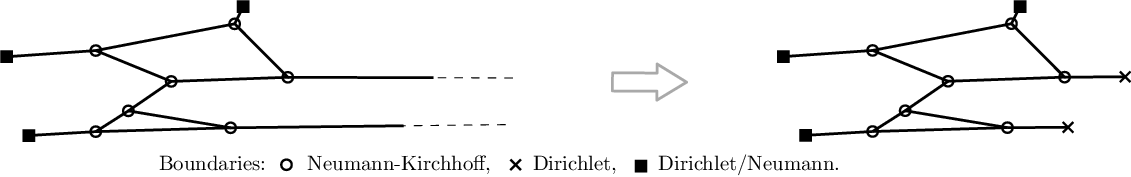}
	\caption{The figure represents an example of definition of the compact graph $\widetilde \Gi(\upvarphi)$ (on 
the right) from a specific infinite graph 
$\Gi(\upvarphi)$ (on the left) composed by $ N_1=11$ finite edges and $\widetilde N_1=2$ infinite edges. We also 
underline 
the boundary conditions characterizing $D(\widetilde A)$ in $\widetilde 
\Gi(\upvarphi)$.}\label{parametrizzazionetadpole2}
\end{figure}

\noindent
Afterwards, for 
every edge 
$e_j$ with $N_1+1\leq j\leq N_1+\widetilde N_1$, we define a ring $\widetilde e_j$ having length $L_j$. We 
consider 
on $L^2(\widetilde e_j,\C)$ a self-adjoint Laplacian $A_j$ with domain $D(A_j)=H^2(\widetilde e_j,\C)$ and we 
denote by $H^s_{\widetilde e_j}:=D(|A_j|^\frac{s}{2})$ for every $s>0.$ On $L^2((0,L_j),\C)$ with $N_1+1\leq j\leq 
N_1+\widetilde N_1$, we consider a 
Dirichlet 
Laplacian $A^\Di_j$ and Neumann Laplacian $A^\NN_j$, while we call, for every $s>0$, 
$$H^s_{e_j^\Di}:=D(|A^\Di_j|^\frac{s}{2}),\ \ \ \ \ \ \  \ \ \ \ \ H^s_{e_j^\NN}:=D(|A^\NN_j|^\frac{s}{2}).$$ 
Now, for every $\psi=(\psi^1,...,\psi^{N_1+\widetilde N_1})\in H^{s_1+s_2}_{\Gi}(\upvarphi)$ with $s_1\in\N$ and 
$s_2\in[0,1/2)$, there exist \begin{equation*}\begin{split}
\begin{cases}\psi_1=(\psi_1^1,...,\psi_1^{N_1+\widetilde N_1})\in H^{s_1+s_2}_{\widetilde 
\Gi(\upvarphi)},   \ \ \ \ \ \  \ \\
f^j\in H^{s_1+s_2}_{\widetilde e_j},\ \ \ &\forall N_1+1\leq j\leq N_1+\widetilde N_1,\\ 
g^j\in 
H^{s_1+s_2}_{e_j^\Di},\ \ \ &\forall N_1+1\leq j\leq N_1+\widetilde N_1,\\
h^j\in H^{s_1+s_2}_{ e_j^\NN},\ \ \ &\forall N_1+1\leq 
j\leq N_1+\widetilde N_1,\end{cases}
\end{split}
\end{equation*}
\begin{equation*}\begin{split}
\text{such that} \ \ \begin{cases}
\psi^j\equiv\psi^j_1,\ \ \ \ \  &\forall 1\leq j\leq N_1,\\
\psi^j(x)=\psi^j_1(x)+g^j(x)+h^j(x),\ \ \ \ &\forall x\in (0,L_j),\ \ \ N_1+1\leq j\leq N_1+\widetilde N_1,\\
\psi^j(x)=f^j\big(x-\big[\frac{x}{L_j
}\big]\big),\ \ \ \ &\forall x\in (L_j,+\infty),\ \ \ N_1+1\leq j\leq N_1+\widetilde N_1.\\
\end{cases}
\end{split}
\end{equation*}
The last decomposition yields that $H^{s_1+s_2}_{\Gi}(\upvarphi)$ can be identified with a suitable subspace of 
$$H^{s_1+s_2}_{\widetilde \Gi(\upvarphi)}\times\prod_{j=N_1+1}^{N_1+\widetilde N_1} \big(H^{s_1+s_2}_{\widetilde 
e_j}\times H^{s_1+s_2}_{e_j^\Di}\times H^{s_1+s_2}_{e_j^\NN}\big) .$$ 
Thanks to the first point of \cite[Proposition\ 4.2]{mio3}, we have
\begin{equation*}\begin{split}
\begin{cases}
H^{s_1+s_2}_{\widetilde \Gi(\upvarphi)}=H^{s_1}_{\widetilde \Gi(\upvarphi)}\cap H^{s_1+s_2}(\widetilde 
\Gi(\upvarphi),\C),\\
H^{s_1+s_2}_{\widetilde e_j}=H^{s_1}_{\widetilde e_j}\cap H^{s_1+s_2}(\widetilde e_j,\C),\ \ \ \ \ \  \ &\forall 
N_1+1\leq 
j\leq N_1+\widetilde N_1,\\
H^{s_1+s_2}_{e_j^\Di}=H^{s_1}_{e_j^{\Di}}\cap H^{s_1+s_2}((0,L_j),\C),\ \ \ \ \ \  \ &\forall N_1+1\leq j\leq 
N_1+\widetilde N_1,\\
H^{s_1+s_2}_{e_j^\NN}=H^{s_1}_{e_j^{\NN}}\cap H^{s_1+s_2}((0,L_j),\C),\ \ \ \ \ \  \ &\forall N_1+1\leq j\leq 
N_1+\widetilde N_1.\\
\end{cases}
\end{split}
\end{equation*}
The last relations imply that, for every $\psi\in H^{s_1+s_2}_{\Gi}(\upvarphi)$ with $s_1\in\N$ and 
$s_2\in[0,1/2)$, there holds $\psi\in H^{s_1}_{\Gi}(\upvarphi)\cap H^{s_1+s_2}_p$ achieving the 
proof of the first point of the proposition. The second and the third statement follow from the same 
techniques by respectively using the second and third point of \cite[Proposition\ 4.2]{mio3}.\qedhere
\end{proof}

In the following theorem, we collect the well-posedness and the controllability result for the bilinear 
Schr\"odinger equation in this general framework. The well-posedness is proved exactly as \cite[Proposition\ 
3.3]{mio5} 
by using Proposition \ref{bor} instead of \cite[Proposition\ 3.2]{mio5}. The controllability result subsequently 
follows from 
the same arguments of \cite[Theorem\ 3.6]{mio5} by considering Proposition \ref{approx} instead of 
\cite[Proposition\ 
B.2]{mio5}.

\begin{teorema}\label{generale} Let 
$\upvarphi:=\{\ffi_k\}_{k\in\N^*}$ be an orthonormal system of $L^2_p$ made by some eigenfunctions of $A$ and 
let 
$\{\mu_{k}\}_{k\in\N^*}$ be the corresponding eigenvalues.

	\smallskip
	\noindent
{\bf 1)} Let the couple $(A,B)$ satisfy Assumptions II$(\upvarphi,\eta,\tilde d)$ with $\eta>0$ and $\tilde d\geq 
0$. Let $d$ be introduced in Assumptions II and $\mu_k\sim k^2$. For every $\psi_0\in H^{2+d}_{\Gi}(\upvarphi)$ 
and 
$u\in L^2((0,T),\R)$ with $T>0$. There
	exists a unique mild solution $\psi\in C_0([0,T],H^{s}_{\Gi}(\upvarphi))$ of the \eqref{mainx1}. In addition, 
the flow of \eqref{mainx1} on $\Hi(\upvarphi)$ can be extended to a unitary flow $\G_t^{u}$ with respect to the 
$L^2_p-$norm such 
that $\G_t^{u}\psi_0=\psi(t)$ for any solution $\psi$ of \eqref{mainx1} with initial data $\psi_0\in 
\Hi(\upvarphi)$.

	\smallskip
	\noindent
	{\bf 2)} If there exist $C>0$ and $\tilde d\geq 0$ such that
$$ |\mu_{k+1}-\mu_k|\geq C  k^{-{\tilde d}},\ \ \ \ \ \ \ \ \ \ \ \ \forall k\in\N^*$$
	 and if $(A,B)$ satisfies Assumptions I$(\upvarphi,\eta)$ and Assumptions II$(\upvarphi,\eta,\tilde d)$ for 
$\eta>0$, then the \eqref{mainx1} is globally exactly controllable in $H^{s}_{\Gi}(\upvarphi)$ for $s=2+d$ with 
$d$ from Assumptions II.
\end{teorema}

\smallskip
\noindent
{\bf Acknowledgments.} 
The second author was financially supported by the ISDEEC project by ANR-16-CE40-0013.

\appendix\section{Global approximate controllability }
Let us denote by $U(\Hi)$ the space of the unitary operators on a Hilbert space $\Hi.$
\begin{defi}
Let 
$\upvarphi:=\{\ffi_k\}_{k\in\N^*}$ be an orthonormal system of $L^2_p$ made by some eigenfunctions of $A$.	The 
\eqref{mainx1} is said to be globally approximately controllable in $H_{\Gi}^{s}(\upvarphi)$ with $s>0$ 
if the following assertion is verified. For every $\epsilon>0$, $\psi\in H^{s}_{\Gi}(\upvarphi)$ and
$\widehat\G\in U(\Hi(\upvarphi))$ such that 
$\widehat\G\psi\in H^{s}_{\Gi}(\upvarphi)$, there exist $T>0$ and 
$u\in L^2((0,T),\R)$ such that
	$$\|\widehat\G\psi-\G^u_T\psi\|_{(s)}<\epsilon.$$
\end{defi}
\begin{prop}\label{approx}
Let 
$\upvarphi:=\{\ffi_k\}_{k\in\N^*}$ be an orthonormal system of $L^2_p$ made by some eigenfunctions of $A$. If the 
hypotheses of Theorem \ref{generale} are satisfied, then the \eqref{mainx1} is globally approximately 
controllable in $H^{s}_{\Gi}(\upvarphi)$ for $s=2+d$ with $d$ from Assumptions II.
\end{prop}
\begin{proof}
The proof is the same of \cite[Proposition\ B.2]{mio5}.\qedhere\end{proof}
\begin{osss}\label{approxT}
Let us consider the framework introduced in Section \ref{sec2} with $\Ti$ an infinite tadpole graph.
	As Proposition $\ref{approx}$, the problem (\ref{mainT}) is globally approximately controllable in 
$H_{\Ti}^{4}(\upvarphi)$ when the hypotheses of Theorem $\ref{globalegirino}$ are verified. Indeed, for every 
$(j,k),(l,m)\in \{(j,k)\in(\N^*)^2:j< k\}$ so that $(j,k)\neq(l,m)$ and such that 
$\mu_j-\mu_k-\mu_j+\mu_m={\pi^2}(j^2-k^2-l^2+m^2)=0,$ there exists $C>0$ such that
	\begin{equation*}\begin{split}
	&\la 
\ffi_j,B\ffi_j\ra_{L^2_p}-\la\ffi_k,B\ffi_k\ra_{L^2_p}-\la\ffi_l,B\ffi_l\ra_{L^2_p}+\la\ffi_m,B\ffi_m\ra_{L^2_p}
=C(j^{-4}-k^{-4}-l^{-4}+m^{4}) \neq 0.
	\end{split}\end{equation*}
	Finally, the arguments leading to Proposition $\ref{approx}$ also ensure the claim.
\end{osss}
\begin{osss}\label{approxS}
Let us consider the framework introduced in Section \ref{sec3} with $\Si$ a star graph composed by a finite 
number of edges of finite or infinite length.
	Equivalently to Remark \ref{approxT}, the (\ref{mainS}) is globally approximately controllable in 
$H_{\Si}^{3}(\upvarphi)$ when the hypotheses of Theorem \ref{globalestella} are verified. Indeed, for every 
$(j,k),(l,m)\in \{(j,k)\in(\N^*)^2:j< k\}$ so that $(j,k)\neq(l,m)$ and such that $\mu_j-\mu_k-\mu_j+\mu_m=0,$ we 
have
	\begin{equation*}\begin{split}
	&\la 
\ffi_j,B\ffi_j\ra_{L^2_p}-\la\ffi_k,B\ffi_k\ra_{L^2_p}-\la\ffi_j,B\ffi_j\ra_{L^2_p}+\la\ffi_m,B\ffi_m\ra_{L^2_p}
\neq0.
	\end{split}\end{equation*}
\end{osss}

\smallskip
\noindent
{\bf Data availability.} Data sharing is not applicable to this article as no new data were created or analyzed in 
this study.

\end{document}